\documentclass[12pt]{amsart}  

\usepackage[latin1]{inputenc}
\usepackage{amsmath} 
\usepackage{amsfonts}
\usepackage{amssymb}
\usepackage{stmaryrd}
\usepackage{latexsym} 
\usepackage{graphicx}
\usepackage{subfigure}
\usepackage{color}
\usepackage{hyperref}
\usepackage{verbatim}
\usepackage[all]{xy}
\usepackage{graphics}
\usepackage{pdfsync}
\usepackage{xcolor}
\usepackage{tikz}
\usepackage{bm}

\usepackage{listings}
\lstdefinestyle{mystyle}{
    commentstyle=\color{codegreen},
    keywordstyle=\color{magenta},
    basicstyle=\ttfamily\footnotesize,
    breakatwhitespace=false,         
    breaklines=true,                 
    captionpos=b,                    
    keepspaces=true,                 
    numbers=left,                    
    numbersep=5pt,                  
    showspaces=false,                
    showstringspaces=false,
    showtabs=false,                  
    tabsize=5
}
\usetikzlibrary{arrows.meta}
\usetikzlibrary{shapes.geometric}

\theoremstyle{definition}
\newtheorem{theorem}{Theorem}[section]
\newtheorem{proposition}[theorem]{Proposition}

\newtheorem{lemma}[theorem]{Lemma}

\newtheorem{conjecture}[theorem]{Conjecture}
\newtheorem{definition}[theorem]{Definition}
\newtheorem{example}[theorem]{Example}
 
\newtheorem{method}[theorem]{Method}

\newtheorem{case}{Case}

\theoremstyle{remark}
\newtheorem{remark}[theorem]{Remark}
\newtheorem{notation}[theorem]{Notation}

\numberwithin{equation}{section}
\usepackage[backend=bibtex]{biblatex}
\begin{document}

\title{Graphs missing a connected partition}

\author{Foster Tom}
\thanks{Department of Mathematics, MIT}

\begin{abstract}
We prove that a graph with a cut vertex whose deletion produces at least five connected components must be missing a connected partition of some type. We prove that this also holds if there are four connected components that each have at least two vertices. In particular, the chromatic symmetric function of such a graph cannot be $e$-positive. This brings us very close to the conjecture by Dahlberg, She, and van Willigenburg of non-$e$-positivity for all trees with a vertex of degree at least four. We also prove that spiders with four legs cannot have an $e$-positive chromatic symmetric function.
\end{abstract}



\maketitle
\vspace{30pt}
\section{Introduction}\label{section:introduction}

Graph partitioning is a fundamentally important problem in computer science. For example, the max-flow min-cut theorem \cite{maxflowmincut} is a cornerstone result in optimization theory. There are numerous practical applications for finding communities in large networks and graph partitions can allow for parallelization of algorithms to reduce computational complexity. Surveys of graph partitioning problems and algorithms can be found in \cite{graphparalg, graphparadv, graphparmeth, graphclustering}.

Let $G=(V,E)$ be an $n$-vertex graph and let $\lambda=(\lambda_1,\ldots,\lambda_\ell)$ be an integer partition of $n$. A \emph{connected partition of $G$ of type $\lambda$} is a partition of the vertices $V$ into connected subsets of sizes $\lambda_1,\ldots,\lambda_\ell$. A more precise definition and examples will be given in Section \ref{section:background}.

Many authors studied problems of finding connected partitions that satisfy certain additional constraints or optimizations \cite{fairdiv, balconpar, balgraphpar, parcon, highlyconpar, doublypar}. In this paper, we consider the problem of finding an integer partition $\lambda$ for which $G$ is \emph{missing} a connected partition of type $\lambda$. One application of this is to non-$e$-positivity of chromatic symmetric functions. Stanley \cite{chromsym} defined the \emph{chromatic symmetric function} of a graph $G=(V,E)$ to be
\begin{equation*}
X_G=\sum_{\substack{\kappa:V\to\{1,2,3,\ldots\}\\\text{ if }ij\in E,\text{ then }\kappa(i)\neq\kappa(j)}}\prod_{v\in V}x_{\kappa(v)}
\end{equation*} 
and we say that $X_G$ is \emph{$e$-positive} if it can be written as a nonnegative linear combination in the basis of \emph{elementary symmetric functions}, defined by
\begin{equation*}
e_\lambda=e_{\lambda_1}\cdots e_{\lambda_\ell},\text{ where }e_k=\sum_{i_1<\cdots<i_k}x_{i_1}\cdots x_{i_k}.
\end{equation*}
It is also common to say that $G$ itself is \emph{$e$-positive}. Wolfgang found a powerful necessary condition for $e$-positivity.

\begin{lemma} \cite[Proposition 1.3.3]{conparts} If a connected graph $G$ is missing a connected partition of some type $\lambda\vdash n$, then it is not $e$-positive.
\end{lemma}

Dahlberg, She, and van Willigenburg \cite{chromspostrees} were interested in determining which trees are $e$-positive. Many authors \cite{eposclawcon, chromspostrees, spider4m22m1, spiderbroom, spiderepos} proved $e$-positivity or non-$e$-positivity for particular subclasses of trees. Dahlberg, She, and van Willigenburg checked the following conjecture for all trees with at most $12$ vertices.

\begin{conjecture}\label{conj:four}
\cite[Conjecture 6.1]{chromspostrees} If $T$ is a tree with a vertex of degree at least $4$, then $T$ is not $e$-positive.
\end{conjecture} 

They proved the following result. 

\begin{theorem}\cite[Theorem 4.1]{chromspostrees} If $G$ is an $n$-vertex graph with a cut vertex whose deletion produces $d\geq\log_2 n+1$ connected components, then $G$ is missing a connected partition of some type $\lambda\vdash n$.
\end{theorem}

In particular, if $T$ is an $n$-vertex tree with a vertex of degree $d\geq\log_2n+1$, then $T$ is not $e$-positive. Zheng strengthened this result considerably.

\begin{theorem}\cite[Theorem 3.13]{spiderepos} If $G$ is a graph with a cut vertex whose deletion produces at least $6$ connected components, then $G$ is missing a connected partition of some type $\lambda\vdash n$.
\end{theorem}

Our main result lowers this bound further.

\begin{theorem}\label{thm:main}
Let $G$ be a graph with a cut vertex whose deletion produces either at least $5$ connected components, or $4$ connected components that each have at least two vertices. Then $G$ is missing a connected partition of some type $\lambda\vdash n$.
\end{theorem}

In particular, if $T$ is a tree with a vertex of degree at least $5$, or with a vertex of degree $4$ that is not adjacent to any leaf, then $T$ is not $e$-positive.

We will show in Proposition \ref{prop:6mhas} that there exists an infinite family of trees with a vertex of degree $4$ that is adjacent to a leaf, and which have a connected partition of every type $\lambda\vdash n$. Therefore, in this sense, Theorem \ref{thm:main} cannot be improved. This means that to prove non-$e$-positivity of trees with a vertex of degree $4$ in general, we need to know more about the tree to calculate particular coefficients. In the specific case of \emph{spiders}, which are graphs $S(\mu)$ formed by joining paths of lengths $\mu_1,\ldots,\mu_\ell$, the \emph{legs} of the spider, at a single vertex, we will prove that Conjecture \ref{conj:four} holds.

\begin{theorem}\label{thm:spider4}
Spiders with four legs are not $e$-positive.
\end{theorem}

\section{Background}\label{section:background}

\begin{figure}
\caption{\label{fig:conpar} A connected partition of $S(6,4,1,1)$ of type $\lambda=(\textcolor{teal}5,\textcolor{magenta}4,\textcolor{cyan}3,\textcolor{orange}1)$}$$
\begin{tikzpicture}
\draw [dashed] (2,0)--(4,0) (6,0)--(7,0);
\node at (-1.5,0) {$S(6,4,1,1)=$};
\filldraw (0,0) [color=cyan] circle (3pt) node[align=center,below] (1){};
\filldraw (1,0) [color=cyan] circle (3pt) node[align=center,below] (2){};
\filldraw (2,0) [color=cyan] circle (3pt) node[align=center,below] (3){};
\filldraw (3,0) [color=orange] circle (3pt) node[align=center,below] (4){};
\filldraw (4,0) [color=teal]circle (3pt) node[align=center,below] (5){};
\filldraw (5,0) [color=teal]circle (3pt) node[align=center,below] (6){};
\filldraw (5.5,0.866) [color=teal] circle (3pt) node[align=center,above] (7){};
\filldraw (6,0) [color=teal] circle (3pt) node[align=center,below] (8){};
\filldraw (6.5,0.866) [color=teal]circle (3pt) node[align=center,above] (9){};
\filldraw (7,0) [color=magenta]circle (3pt) node[align=center,below] (10){};
\filldraw (8,0) [color=magenta]circle (3pt) node[align=center,below] (11){};
\filldraw (9,0) [color=magenta]circle (3pt) node[align=center,below] (12){};
\filldraw (10,0) [color=magenta]circle (3pt) node[align=center,below] (13){};
\draw [color=cyan](0,0)--(2,0);
\draw [color=teal] (4,0)--(6,0) (6,0)--(5.5,0.866) (6,0)--(6.5,0.866);
\draw[color=magenta] (7,0)--(10,0);
\end{tikzpicture}$$\end{figure}

A \emph{graph} $G=(V,E)$ consists of a set $V$ of \emph{vertices} and a set $E$ of \emph{edges}, which are sets of two vertices. In particular, we do not allow \emph{loops}, which are edges joining a vertex to itself, or \emph{parallel edges}, which are edges joining the same pair of vertices. We will always use $n$ to denote the \emph{order} of $G$, which is the number of vertices. For a subset $S\subseteq V$, the \emph{induced subgraph} is $G[S]=(S,\{e=vw\in E: \ v,w\in S\})$.

We say that $G$ is \emph{connected} if for every $v,w\in V$, there is a \emph{path} from $v$ to $w$, which is a sequence of vertices $(v=v_1,v_2\ldots,v_{k-1},v_k=w)$ with every $v_iv_{i+1}\in E$. A \emph{connected component} of $G$ is a maximal subset $S\subseteq V$ such that $G[S]$ is connected. A vertex $v\in V$ of a connected graph is a \emph{cut vertex} if the graph $G\setminus v$ obtained by deleting $v$ and all incident edges is not connected.

A \emph{tree} $T=(V,E)$ is a minimally connected graph, meaning that for every $e\in E$, the graph $(V,E\setminus \{e\})$ is not connected. The \emph{degree} of a vertex $v$ is the number of edges incident to it and a \emph{leaf} of a tree is a vertex of degree one. Note that in a tree, every vertex is a leaf or a cut vertex. If $v$ is a cut vertex of degree $d$, then $T\setminus v$ has $d$ connected components.

A \emph{composition of $n$} is a sequence of positive integers $\alpha=(\alpha_1,\ldots,\alpha_\ell)$ with sum $n$. An \emph{(integer) partition of $n$} is a weakly decreasing composition $\lambda=(\lambda_1,\ldots,\lambda_\ell)$ of $n$. These integers $\lambda_i$ are called the \emph{parts} of $\lambda$. We write $\lambda\vdash n$ to mean that $\lambda$ is an integer partition of $n$. We now define the main object of study in this paper.
\begin{definition}
Let $G=(V,E)$ be an $n$-vertex graph and let $\lambda=(\lambda_1,\ldots,\lambda_\ell)\vdash n$. A \emph{connected partition of $G$ of type $\lambda$} is a set partition $\mathcal S=\{S_1,\ldots,S_\ell\}$ of $V$ so that
\begin{itemize}
\item the induced subgraph $G[S_i]$ is connected for every $1\leq i\leq\ell$, and
\item $|S_i|=\lambda_i$ for every $1\leq i\leq\ell$.
\end{itemize} 
We say that $G$ is \emph{missing} a connected partition of type $\lambda$ if no such $\mathcal S$ exists.
\end{definition}

\begin{example}
The \emph{spider graph} $G=S(6,4,1,1)$ is shown in Figure \ref{fig:conpar}. We have indicated a connected partition of $G$ of type $\lambda=(5,4,3,1)\vdash 13$. We will see in Proposition \ref{prop:6mhas} that $G$ has a connected partition of every type $\lambda\vdash 13$. But by Theorem \ref{thm:spider4}, $G$ still fails to be $e$-positive. This shows that Wolfgang's necessary condition for $e$-positivity is not sufficient.
\end{example}

\begin{example}
If $n$ is even, a connected partition of $G=(V,E)$ of type $\lambda=(2,2,\ldots,2,2)$ is equivalent to a \emph{perfect matching} of $G$, which is a subset $U\subseteq E$ such that every vertex is in exactly one edge of $U$.
\end{example}

\section{Proof of main result}\label{section:main}

\begin{figure}\caption{\label{fig:graphG} A graph $G$ with a cut vertex $v$ joining at least $3$ connected components}
\vspace{5pt}
\begin{tikzpicture}
\filldraw (0,0) circle (3pt) node[align=center,below] (1){$v$};
\node[ellipse, minimum height=10mm,minimum width=20mm,draw] at (2,0) {$b$};
\draw (-1,1) circle (10pt) node {$c_1$};
\draw (0,1) circle (10pt) node {$\ldots$};
\draw (1,1) circle (10pt) node {$c_k$};
\node[ellipse, minimum height=10mm,minimum width=20mm,draw] at (-2,0) {$a$};
\node[ellipse, minimum height=10mm, minimum width=40mm,draw, dotted] at (0,1){};
\draw [dotted] (2,1)--(2.5,1);
\draw node at (2.5,1) {$c$};
\draw node at (-4,0) {$G=$};
\draw (0,0)--(1,0) (0,0)--(1.3,0.35) (0,0)--(1.3,-0.35) (0,0)--(-1,0) (0,0)--(-1.3,0.35) (0,0)--(-1.3,-0.35) (0,0)--(0,0.66) (0,0)--(0.1,0.68) (0,0)--(-0.1,0.68) (0,0)--(0.75,0.75) (0,0)--(0.66,0.84) (0,0)--(0.85,0.66) (0,0)--(-0.75,0.75) (0,0)--(-0.66,0.84) (0,0)--(-0.85,0.66);
\end{tikzpicture}
\end{figure}

We first fix notation that we will use throughout this section.

\begin{notation}\label{notation}
Let $G$ be an $n$-vertex connected graph with a cut vertex $v$ whose deletion produces at least $3$ connected components, and denote their orders
\begin{equation*}
a\geq b\geq c_1\geq\cdots\geq c_k\geq 1,\text{ where }k\geq 1.
\end{equation*} See Figure \ref{fig:graphG}. Let $c=c_1+\cdots+c_k$. We will always assume that $c\geq 2$. Note that
\begin{equation*}
n=a+b+c+1\geq 2b+c+1.
\end{equation*}
We will always use $\lambda$ to denote an integer partition of $n$. 
\end{notation}
We now restate our main result using this notation.

\begin{theorem}\label{thm:main2}
Let $G$, $a$, $b$, $c$, $c_1$, and $n$ be as in Notation \ref{notation}. Suppose that $c\geq 2$. If any of the following hold, then $G$ is missing a connected partition of some type $\lambda\vdash n$.
\begin{enumerate}
\item We have $b\leq 2c-2$.
\item We have $b=2c-1$ and $c\geq c_1+1$.
\item We have $2c\leq b\leq\frac{c^2}2$.
\item We have $b\geq\frac{c^2}2$ and $c\geq c_1+2$.
\end{enumerate}
In particular, if $c\geq c_1+2$, then no matter which range $b$ lies in, we can conclude that $G$ is missing a connected partition of some type $\lambda\vdash n$.
\end{theorem}

The proof of Theorem \ref{thm:main2} is structured as follows. We begin by reformulating connected partitions of type $\lambda$ in terms of partial sums of rearrangements of $\lambda$ (Lemma \ref{lem:partsums}), which we use to prove Part (4). We then describe a strategy for choosing the parts of $\lambda$ (Lemma \ref{lem:q}), which will prove Parts (1) and (2). To prove Part (3), we will need a technical result to establish an inequality (Lemma \ref{lem:analysis}). Our estimates will only be valid if $c$ is sufficiently large, namely $c\geq 500$, but we can check the cases with $2\leq c\leq 499$ by computer (Lemma \ref{lem:smallc}). Sage code used to check these cases is provided in the Appendix. 

\begin{remark} The condition that $c\geq c_1+2$ is precisely that the deletion of $v$ produces either at least $5$ connected components, or $4$ connected components that each have at least two vertices, so Theorem \ref{thm:main2} implies Theorem \ref{thm:main}. We cannot remove this hypothesis because the spider graph $S(6,4,1,1)$ from Figure \ref{fig:conpar} has a connected partition of every type $\lambda\vdash 13$. Incidentally, it will follow from Theorem \ref{thm:spider4} that the graph $S(6,4,1,1)$ is not $e$-positive. \end{remark}

\begin{remark} The condition that $c\geq c_1+1$ is precisely that the deletion of $v$ produces at least $4$ connected components. We cannot remove this hypothesis because the spider graph $S(5,3,2)$ has a connected partition of every type $\lambda\vdash 11$ and is in fact $e$-positive.\end{remark}

\begin{remark}
We cannot remove the hypothesis that $c\geq 2$ because the family of spider graphs $S(m+1,m,1)$ all have a connected partition of every type $\lambda\vdash 2m+3$ and are in fact $e$-positive \cite[Theorem 3.2]{eposclawcon}.
\end{remark}

We begin by describing another way to think about connected partitions. A composition $\alpha=(\alpha_1,\ldots,\alpha_\ell)$ is a \emph{rearrangement of $\lambda$} if each positive integer occurs the same number of times in $\alpha$ and in $\lambda$. We denote by
\begin{equation*}
\text{set}(\alpha)=\{\alpha_1,\alpha_1+\alpha_2,\ldots,\alpha_1+\cdots+\alpha_{\ell-2},\alpha_1+\cdots+\alpha_{\ell-1}\}
\end{equation*}
the set of (proper) \emph{partial sums} of $\alpha$.

\begin{example}
The composition $\alpha=(3,1,5,4)$ is a rearrangement of $\lambda=(5,4,3,1)$ and has set of partial sums $\text{set}(\alpha)=\{3,4,9\}$.
\end{example}

\begin{lemma}\label{lem:partsums} Let $G$, $a$, $b$, $c$, $c_1$, $n$, and $v$ be as in Notation \ref{notation}. Let $\lambda$ be a partition of $n$ with all parts at least $(c_1+1)$. If every rearrangement $\alpha$ of $\lambda$ has a partial sum $s\in\text{set}(\alpha)$ that lies in the interval $I=\{b+1,\ldots,b+c\}$, then $G$ is missing a connected partition of type $\lambda$. 
\end{lemma}

\begin{proof}
Suppose that $G$ has a connected partition $\mathcal S$ of type $\lambda$. We will find a rearrangement $\alpha$ of $\lambda$ that has no partial sum in $I$. Let $A,B,C_1,\ldots,C_k$ be the connected components of $G\setminus v$ whose orders are respectively $a\geq b\geq c_1\geq \cdots\geq c_k$ and let $S\in\mathcal S$ be the subset with $v\in S$. If any $C_i\nsubseteq S$, then letting $u\in C_i\setminus S$, the subset $S'\in\mathcal S$ with $u\in S'$ would have $|S'|\leq c_i\leq c_1$, which is impossible because every part of $\lambda$ is at least $(c_1+1)$. So every $C_i\subseteq S$. This means that $\mathcal S$ is of the form
$\mathcal S=\{A_1,\ldots,A_i,S,B_1,\ldots,B_j\}$
where $A_{i'}\subseteq A$ for $1\leq i'\leq i$ and $B_{j'}\subseteq B$ for $1\leq j'\leq j$. We now take $\alpha$ to be the composition
\begin{equation*}
\alpha=(|B_1|,\ldots,|B_j|,|S|,|A_1|,\ldots,|A_i|).
\end{equation*}
By definition, $\alpha$ is a rearrangement of $\lambda$, and we have
\begin{align*}
\alpha_1+\cdots+\alpha_j&=|B_1|+\cdots+|B_j|\leq |B|=b\text{ and }\\
\alpha_1+\cdots+\alpha_{j+1}&=|B_1|+\cdots+|B_j|+|S|\geq|B|+|C_1|+\cdots+|C_k|+1=b+c+1.
\end{align*}
Therefore, no partial sum of $\alpha$ lies in the interval $I=\{b+1,\ldots,b+c\}$.
\end{proof}

We also use the following result, the two-coin solution to the Frobenius coin problem, which gives a sufficient condition to find a partition $\lambda$ of $n$ with specified parts. 

\begin{theorem} \cite{frobeniuscoin} \label{thm:frobenius} Let $x$ and $y$ be positive integers with $\gcd(x,y)=1$. Then every integer $n\geq (x-1)(y-1)$ can be written in the form $n=a_1x+a_2y$ for some integers $a_1,a_2\geq 0$. 
\end{theorem}

We can now prove Part (4) of Theorem \ref{thm:main2}, the case where $c\geq c_1+2$ and $b\geq\frac{c^2}2$.

\begin{proof}[Proof of Part (4) of Theorem \ref{thm:main2}. ]
The idea is to take a partition $\lambda$ of $n$ with parts $c$ or $(c-1)$. Because $b\geq\frac{c^2}2$, we have
\begin{equation*}
n\geq 2b+c+1\geq c^2+c+1\geq (c-1)(c-2),
\end{equation*} 
so by Theorem \ref{thm:frobenius}, we can write $n=a_1c+a_2(c-1)$ and take $\lambda$ to be $a_1$ $c$'s followed by $a_2$ $(c-1)$'s. All parts of $\lambda$ are at least $c-1\geq c_1+1$. Let $\alpha$ be a rearrangement of $\lambda$. Because $c\geq c_1+2\geq 3$, we have $\alpha_1\leq c\leq\frac{c^2}2\leq b$, so there exists some maximal $i$ with $\alpha_1+\cdots+\alpha_i\leq b$, and then the partial sum $\alpha_1+\cdots+\alpha_i+\alpha_{i+1}$ satisfies
\begin{equation*}
b+1\leq\alpha_1+\cdots+\alpha_i+\alpha_{i+1}\leq b+\alpha_{i+1}\leq b+c.
\end{equation*}
Therefore the result follows from Lemma \ref{lem:partsums}.
\end{proof}

We now describe another way to ensure that all rearrangements of $\lambda$ have a partial sum in the interval $I=\{b+1,\ldots,b+c\}$.

\begin{lemma}\label{lem:q} Let $G$, $a$, $b$, $c$, $c_1$, and $n$ be as in Notation \ref{notation}. Let $q$ be a positive integer and let $J=\{x,x+1,\ldots,y-1,y\}$, where
\begin{equation}
x=\left\lceil\frac{b+1}q\right\rceil\text{ and }y=\left\lfloor\frac{b+c}q\right\rfloor.
\end{equation}
Suppose that $x\geq c_1+1$. If $\lambda$ is a partition of $n$ with all parts in the interval $J$, then $G$ is missing a connected partition of type $\lambda$.
\end{lemma} 

\begin{proof}
All parts of $\lambda$ are at least $x\geq c_1+1$ and every rearrangement $\alpha$ of $\lambda$ has
\begin{equation*}
b+1\leq qx\leq\alpha_1+\cdots+\alpha_q\leq qy\leq b+c,
\end{equation*}
so the result follows from Lemma \ref{lem:partsums}.
\end{proof}

We now generalize the result of the Frobenius coin problem by seeing which numbers can be written as a sum of numbers in an interval $J$.

\begin{lemma}\label{lem:frobint}
Let $x$ and $y$ be positive integers with $x<y$. Then there exists a partition $\lambda$ of $n$ with all parts in the interval $J=\{x,\ldots,y\}$ as long as 
\begin{equation}\label{eq:intervalcoin}
n\geq\left\lceil\frac{x-1}{y-x}\right\rceil x.
\end{equation}
\end{lemma}

\begin{proof}
For a positive integer $t$, let $tJ=\{tx,tx+1,\ldots,ty-1,ty\}$. First note that every integer $n\in tJ$ has a partition $\lambda$ consisting of $t$ integers in $J$. Indeed, for $n=ty$ we can take $\lambda=(y,\ldots,y)$, and for $tx\leq n\leq ty-1$ we can write $n=tq+r$ for some integers $x\leq q\leq y-1$ and $0\leq r\leq t-1$, and take $\lambda$ to be $r$ $(q+1)$'s followed by $(t-r)$ $q$'s.\\

We now see that as $t$ increases, the intervals $tJ$ become closer to each other. Specifically, for $t\geq\left\lceil\frac{x-1}{y-x}\right\rceil\geq\frac{x-1}{y-x}$, we have $ty+1\geq (t+1)x$, which means that as soon as we exit the interval $tJ$, we are already in the interval $(t+1)J$. This means that every integer $n$ satisfying \eqref{eq:intervalcoin} is in some interval $tJ$ and has a partition $\lambda$ with all parts in $J$.
\end{proof}


\begin{example}
Let $x=7$ and $y=9$, so that $J=\{7,8,9\}$ and we can write the numbers
\begin{equation*}
\{\underbrace{7,8,9}_J\}\cup\{\underbrace{14,15,16,17,18}_{2J}\}\cup\{\underbrace{21,22,23,24,25,26,27}_{3J}\}\cup\{\underbrace{28,29,30,31,32,33,34,35,36}_{4J}\}\cup\cdots
\end{equation*}
as a sum of numbers in $J$. We see that once $t\geq\left\lceil\frac{x-1}{y-x}\right\rceil=3$, we no longer have a gap between the intervals $tJ$ and $(t+1)J$, so we are able to express any integer $n\geq 3x=21$.
\end{example}

\begin{lemma}\label{lem:nb2c} Let $G$, $a$, $b$, $c$, $c_1$, and $n$ be as in Notation \ref{notation}. If $n\geq\left\lceil\frac b{c-1}\right\rceil(b+1)$, then $G$ is missing a connected partition of some type $\lambda\vdash n$. 
\end{lemma}

\begin{proof}
By Lemma \ref{lem:frobint}, there exists a partition $\lambda$ of $n$ with all parts in the interval $J=\{b+1,\ldots,b+c\}$. Now by Lemma \ref{lem:q} with $q=1$, $G$ is missing a connected partition of type $\lambda$.
\end{proof}

We can now prove Part (1) of Theorem \ref{thm:main2}, the case where $b\leq 2c-2$; and Part (2) of Theorem \ref{thm:main2}, the case where $c\geq c_1+1$ and $b=2c-1$. 

\begin{proof}[Proof of Part (1) of Theorem \ref{thm:main2}. ] Because $b\leq 2c-2$, we have
\begin{equation*}
\left\lceil\frac b{c-1}\right\rceil (b+1)\leq 2(b+1)\leq 2b+c+1\leq n,
\end{equation*}
so the result follows from Lemma \ref{lem:nb2c}. 
\end{proof}

\begin{proof}[Proof of Part (2) of Theorem \ref{thm:main2}. ]
First note that if $c=2$, then $b=3$, and taking $\lambda=(2,\ldots,2)$ if $n$ is even or $\lambda=(3,2,\ldots,2)$ if $n$ is odd, every rearrangement $\alpha$ has the partial sum $\alpha_1+\alpha_2$ in the interval $I=\{4,5\}$ and we are done by Lemma \ref{lem:partsums}. Now suppose that $c\geq 3$. Let $x=\left\lceil\frac{b+1}2\right\rceil=c\geq c_1+1$ and $y=\left\lfloor\frac{b+c}2\right\rfloor=\left\lfloor\frac{3c-1}2\right\rfloor$. By Lemma \ref{lem:frobint}, we can find a partition $\lambda$ of $n$ with all parts in the interval $J=\{x,\ldots,y\}$ because
\begin{equation*}
\left\lceil\frac{x-1}{y-x}\right\rceil x\leq\left\lceil\frac{c-1}{\frac{c-2}2}\right\rceil c=\left\lceil 2+\frac 2{c-2}\right\rceil c\leq 4c\leq 2b+c+1\leq n. 
\end{equation*} 
By Lemma \ref{lem:q} with $q=2$, $G$ is missing a connected partition of type $\lambda$.
\end{proof}

Our last case to prove is Part (3) of Theorem \ref{thm:main2}, the case where $2c\leq b\leq\frac{c^2}2$. Our strategy is as follows. Given integers $b$, $c$, and $n$ with $c\geq 2$, $2c\leq b\leq\frac{c^2}2$, and $n\geq 2b+c+1$, we wish to find a positive integer $q$ such that the integers 
\begin{equation}\label{eq:xy}
x=\left\lceil\frac{b+1}q\right\rceil\text{ and }y=\left\lfloor\frac{b+c}q\right\rfloor
\end{equation} satisfy that $x\geq c+1\geq c_1+1$, $x<y$, and that there exists a partition $\lambda$ of $n$ with all parts in the interval $J=\{x,\ldots,y\}$. Then $G$ will have no connected partition of type $\lambda$ by Lemma \ref{lem:q}. By Lemma \ref{lem:frobint}, a sufficient condition to find such a $\lambda\vdash n$ is if \begin{equation}\label{eq:analysisbound}
\left\lceil\frac{x-1}{y-x}\right\rceil x\leq 2b+c+1.\end{equation}

\begin{example}\label{ex:oneq}
Let $b=75$ and $c=14$, so that $n\geq 2b+c+1=165$. 
\begin{itemize}
\item For $q\geq 6$, we have $x\leq\left\lceil\frac{76}6\right\rceil=13\leq c$, so this will not work.
\item For $q=5$, we have $x=\left\lceil\frac{76}5\right\rceil=16$ and $y=\left\lfloor\frac{89}5\right\rfloor=17$. We have $x\geq c+1$, but 
\begin{equation*}
\left\lceil\frac{x-1}{y-x}\right\rceil x=240>165,
\end{equation*}
so this will not work. We cannot write $n=239$ as a sum of $16$'s and $17$'s.
\item For $q=4$, we have $x=\left\lceil\frac{76}4\right\rceil=19$ and $y=\left\lfloor\frac{89}4\right\rfloor=22$. We have $x\geq c+1$ and 
\begin{equation*}
\left\lceil\frac{x-1}{y-x}\right\rceil x=114\leq 165,
\end{equation*} so this will work. Lemma \ref{lem:frobint} tells us that there exists a partition $\lambda$ of $n$ with all parts in the interval $J=\{19,20,21,22\}$ and Lemma \ref{lem:q} tells us that $G$ will be missing a connected partition of type $\lambda$. 
\end{itemize}
\end{example}

The next example shows that there may be no single value of $q$ for which the inequality \eqref{eq:analysisbound} holds. However, for every value of $n\geq 2b+c+1$, there is some $q$ such that there exists $\lambda\vdash n$ with all parts in the interval $J=\{x,\ldots,y\}$, where $x$ and $y$ are as in \eqref{eq:xy}.

\begin{example}\label{ex:moreq}
Let $b=24$ and $c=7$, so that $n\geq 2b+c+1=56$. 
\begin{itemize}
\item For $q\geq 4$, we have $x\leq 7\leq c$, so this will not work.
\item For $q=3$, we have $x=9$, $y=10$, and $\left\lceil\frac{x-1}{y-x}\right\rceil x=72>56$, so this will not work.
\item For $q=2$, we have $x=13$, $y=15$, and $\left\lceil\frac{x-1}{y-x}\right\rceil x=78>56$, so this will not work.
\item For $q=1$, we have $x=25$, $y=31$, and $\left\lceil\frac{x-1}{y-x}\right\rceil x=100>56$, so this will not work.
\end{itemize}
We find that no single value of $q$ satisfies \eqref{eq:analysisbound}. For $n\geq 72$, we can take $q=3$, but this will not work for every $n\geq 56$. Nevertheless, for each $56\leq n\leq 71$, there is some $q$ such that there exists $\lambda\vdash n$ with all parts in the interval $J=\{x,\ldots,y\}$, where $x$ and $y$ are as in \eqref{eq:xy}.
\begin{align*}
q=3&:\{\underbrace{54,55,\textcolor{red}{56},\textcolor{red}{57},\textcolor{red}{58},\textcolor{red}{59},\textcolor{red}{60}}_{6J}\}\cup\{\underbrace{\textcolor{red}{63},\textcolor{red}{64},\textcolor{red}{65},\textcolor{red}{66},\textcolor{red}{67},\textcolor{red}{68},\textcolor{red}{69},\textcolor{red}{70}}_{7J}\}\cup\{\underbrace{\textcolor{red}{72},\textcolor{red}{73},\textcolor{red}{74},\textcolor{red}{\ldots}}_{8J}\}\cup\cdots\\
q=2&:\{\underbrace{52,53,54,55,56,57,58,59,60}_{4J}\}\cup\{\underbrace{65,66,67,68,69,70,\textcolor{red}{71},72,73,74,75}_{5J}\}\cup\cdots\\
q=1&:\{\underbrace{50,51,52,53,54,55,56,57,58,59,60,\textcolor{red}{61},\textcolor{red}{62}}_{2J}\}\cup\{\underbrace{75,76,77,\ldots}_{3J}\}\cup\cdots
\end{align*}
\end{example}

We will soon show by an analytic argument that if $c\geq 500$ and $2c\leq b\leq\frac{c^2}2$, then we can find a positive integer $q$ such that the integers $x$ and $y$ as in \eqref{eq:xy} satisfy $x\geq c+1$, $x<y$, and the inequality \eqref{eq:analysisbound}. Our estimates will not be valid for $2\leq c\leq 499$, so we check these cases by computer. For $41\leq c\leq 499$, we can always find a single value of $q$ for which \eqref{eq:analysisbound} holds, as in Example \ref{ex:oneq}. When $2\leq c\leq 40$, we will sometimes need different values of $q$ depending on $n$ to find a suitable partition $\lambda$ of $n$, as in Example \ref{ex:moreq}. Sage code used to check the following Lemma will be provided in the Appendix.

\begin{lemma}\label{lem:smallc}
Fix positive integers $b$, $c$, and $n$ with $2\leq c\leq 499$, $2c\leq b\leq \frac{c^2}2$, and $n\geq 2b+c+1$. There exists a positive integer $q$ such that the integers $x$ and $y$ as in \eqref{eq:xy} satisfy that $x\geq c+1$, $x<y$, and there exists a partition $\lambda$ of $n$ with all parts in the interval $J=\{x,\ldots,y\}$.
\end{lemma}

\begin{proof}
First suppose that $2\leq c\leq 40$. By Lemma \ref{lem:frobint}, we can take $q=1$ if $n\geq\left\lceil\frac b{c-1}\right\rceil(b+1)$, so there are finitely many triples $(b,c,n)$ to consider. For each one, we checked by computer, using the Sage code in Method \ref{method:40}, that there exists such a positive integer $q$.

Now suppose that $41\leq c\leq 499$. By Lemma \ref{lem:frobint}, there exists such a partition $\lambda$ if the inequality \eqref{eq:analysisbound} holds. There are finitely many pairs $(b,c)$ to consider. For each one, we checked by computer, using the Sage code in Method \ref{method:500}, that there exists such a positive integer $q$.
\end{proof}

We now come to the most technical part of the proof, where given any $c\geq 500$ and $2c\leq b\leq\frac{c^2}2$, we will find an integer $q$ such that the integers $x$ and $y$ as in \eqref{eq:xy} satisfy $x\geq c+1$, $x<y$, and the inequality 
\begin{equation}
\left\lceil\frac{x-1}{y-x}\right\rceil x\leq 2b+c+1.
\end{equation} We first make a remark about this inequality.

\begin{remark}
If we ignore the floors, ceilings, and ones, we have a rough approximation
\begin{equation*}
\left\lceil\frac{x-1}{y-x}\right\rceil x\approx \frac{b/q}{c/q}\frac bq=\frac{b^2}{cq}.
\end{equation*}
To minimize this, we want $q$ to be large, and to have $x\geq c+1$, we could take $q\approx \frac bc$, which would give us $\left\lceil\frac{x-1}{y-x}\right\rceil x\approx b$, so this is promising. However, the floors and ceilings mean that the denominator $(y-x)$ may be closer to $(\frac cq-2)$, and this $(-2)$ term can really hurt us if $q$ is close to $\frac c2$. Therefore, when $b$ is close to $\frac{c^2}2$, we will need to make a clever choice of $q$.
\end{remark}

\begin{lemma}\label{lem:analysis}
Fix positive integers $b$ and $c$ with $c\geq 500$ and $2c\leq b\leq\frac{c^2}2$. There exists a positive integer $q$ such that, letting
\begin{equation}
x=\left\lceil\frac{b+1}q\right\rceil\text{ and }y=\left\lfloor\frac{b+c}q\right\rfloor,
\end{equation} we have $x\geq c+1$, $x<y$, and the inequality
\begin{equation}
\left\lceil\frac{x-1}{y-x}\right\rceil x\leq 2b+c+1.
\end{equation}
\end{lemma}

\begin{proof}
We first note that
\begin{equation*}
y-x=\left\lfloor\frac{b+c}q\right\rfloor-\left\lceil\frac{b+1}q\right\rceil\geq\left(\frac{b+c}q-\frac{q-1}q\right)-\left(\frac{b+1}q+\frac{q-1}q\right)=\frac cq-2+\frac 1q\geq \frac cq-2
\end{equation*}
because the floors and ceilings round by at most $\frac{q-1}q$. We also have
\begin{equation*}
\left\lceil\frac{x-1}{y-x}\right\rceil x=\left\lceil\frac{\left\lceil\frac{b+1}q\right\rceil-1}{y-x}\right\rceil\left\lceil\frac{b+1}q\right\rceil\leq\left(\frac {b/q}{y-x}+1\right)\left(\frac bq+1\right).
\end{equation*}
Now we will have a few different cases, depending on the value of $b$.

\begin{case} $2c\leq b\leq \frac{c^2}{20}$. We let $q=\left\lfloor\frac bc\right\rfloor$, which means that $\frac bc-1\leq q\leq \frac bc$. Then
\begin{equation*}
x=\left\lceil\frac{b+1}q\right\rceil\geq\left\lceil\frac{b+1}bc\right\rceil\geq c+1.
\end{equation*}
Because $b\geq 2c$, we have $q\geq 2$ and therefore
\begin{equation*}
q\geq\frac 23(q+1)\geq \frac 23\frac bc.
\end{equation*} 
Because $b\leq \frac{c^2}{20}$, we have $q\leq\frac bc\leq\frac c{20}$, so $y-x\geq\frac cq-2\geq 18$ and 
\begin{equation*}
y-x\geq\frac {18}{20}(y-x+2)\geq\frac{18}{20}\frac cq.
\end{equation*}
Now using that $c\geq 500$, we have
\begin{align*}
\left\lceil\frac{x-1}{y-x}\right\rceil x&\leq\left(\frac {b/q}{y-x}+1\right)\left(\frac bq+1\right)\leq\left(\frac{20}{18}\frac bc+1\right)\left(\frac 32c+1\right)\leq(1.12\frac bc+1)(1.51c)\\&\leq 1.7b+0.51c+c\leq 1.7b+0.255b+c=1.955b+c\leq 2b+c+1.
\end{align*}
\end{case}

\begin{case} $\frac{c^2}{20}\leq b\leq\frac{c^2}{4.2}$. We let $q=\left\lfloor\frac bc\right\rfloor$ again, but our estimates will look a bit different. We still have $x\geq c+1$ as before. Because $b\geq\frac{c^2}{20}\geq 25c$, we have $q\geq 25$ and therefore
\begin{equation*}
q\geq \frac {25}{26}(q+1)\geq\frac{25}{26}\frac bc.
\end{equation*}
Because $b\leq\frac{c^2}{4.2}$, we have $q\leq\frac bc\leq\frac c{4.2}$. Now since $y-x$ is an integer and $y-x\geq\frac cq-2\geq 2.2$, we must have $y-x\geq 3$ and
\begin{equation*}
y-x\geq\frac 35(y-x+2)\geq \frac 35\frac cq.
\end{equation*}
Now using that $c\geq 500$, we have
\begin{align*}
\left\lceil\frac{x-1}{y-x}\right\rceil x&\leq\left(\frac{b/q}{y-x}+1\right)\left(\frac bq+1\right)\leq\left(\frac 53\frac bc+1\right)\left(\frac{26}{25}c+1\right)\leq(1.67\frac bc+1)(1.05c)\\&\leq 1.76b+0.05c+c\leq 1.76b+0.002b+c=1.762b+c\leq 2b+c+1.
\end{align*}
\end{case}

\begin{case}$\frac{c^2}{4.2}\leq b\leq\frac{c^2}{3.4}$. We let $q=\lfloor 0.46\sqrt b\rfloor$. Note that $0.487\leq \frac 1{\sqrt{4.2}}\leq\frac{\sqrt b}c\leq\frac 1{\sqrt{3.4}}\leq 0.543$. This means that
\begin{equation*}
x=\left\lceil\frac{b+1}q\right\rceil\geq\frac bq\geq \frac b{0.46\sqrt b}>\frac{\sqrt b}{0.487}\geq c,
\end{equation*}
so $x\geq c+1$ because $x$ is an integer. We also have
\begin{equation*}
c\geq\sqrt{3.4}\sqrt b\geq\frac{\sqrt{3.4}}{0.46}q\geq 4.008q,
\end{equation*}
so because $y-x$ is an integer and $y-x\geq\frac cq-2>2.008$, we must have $y-x\geq 3$ and
\begin{equation*}
y-x\geq\frac 35(y-x+2)\geq\frac 35\frac cq.
\end{equation*}
Now using that $\sqrt b\geq 0.487c\geq 0.487\cdot 500\geq 240$, we have $q\geq 0.46\sqrt b-1\geq 0.4558\sqrt b$ and
\begin{align*}
\left\lceil\frac{x-1}{y-x}\right\rceil x&\leq\left(\frac53\frac bc+1\right)\left(\frac bq+1\right)\leq\left(\frac 5{3\sqrt{3.4}}\sqrt b+1\right)\left(\frac{\sqrt b}{0.4558}+1\right)\\&\leq(0.904\sqrt b+1)(2.194\sqrt b+1)\leq (0.9082\sqrt b)(2.1982\sqrt b)\leq 1.997b\leq 2b+c+1.
\end{align*}
\end{case}

\begin{case}$\frac{c^2}{3.4}\leq b\leq\frac{c^2}2$. Let $q_0=\left\lfloor\frac{\sqrt b}{\sqrt{3.5}}\right\rfloor$. We will either take $q=q_0$ or $q=q_0-1$. We first show that for one of these values of $q$, we will have $y-x\geq 2$. Note that $\frac 1{\sqrt{3.4}}\leq\frac{\sqrt b}c\leq\frac 1{\sqrt 2}$, so
\begin{equation*}
c\geq\sqrt 2\sqrt b\geq\sqrt 2\sqrt{3.5}q_0\geq 2.64q_0.
\end{equation*} Also note that because $c\geq 500$, we have $\sqrt b\geq\frac{500}{\sqrt{3.4}}\geq 270$ and $q_0\geq\frac{270}{\sqrt{3.5}}-1\geq 140$. Let us first consider $q=q_0$. Use integer division to write $b+1=c_0q_0+r_0$ for some integers $c_0$ and $0\leq r_0\leq q_0-1$. If 
\begin{equation*}r_0\geq 0.38q_0\geq 0.37q_0+1,\end{equation*}
this will mean that $c+r_0-1\geq 2.64q_0+0.37q_0=3.01q_0>3q_0$. Therefore,
\begin{equation*}
b+c=b+1+c-1=c_0q_0+(c+r_0-1)>(c_0+3)q_0,
\end{equation*}
so at least three multiples of $q_0$, namely $(c_0+1)q_0$, $(c_0+2)q_0$, and $(c_0+3)q_0$, lie strictly between $b+1$ and $b+c$. This means that $x\leq c_0+1$ and $y\geq c_0+3$, so $y-x\geq 2$.\\

Now if $r_0<0.38q_0$, we will consider $q=q_0-1$. We have $\sqrt{3.5}q_0\leq\sqrt b\leq\sqrt{3.5}(q_0+1)$, so
\begin{equation*}
3.5q_0^2\leq b\leq 3.5q_0^2+7q_0+3.5.
\end{equation*}
Now the integer $c_0$ satisfies $b+1-q_0\leq c_0q_0\leq b+1$, so using that $q_0\geq 140$, we have
\begin{equation*}
3.49q_0^2\leq 3.5q_0^2-q_0\leq b+1-q_0\leq c_0q_0\leq b+1\leq 3.5q_0^2+7q_0+4.5\leq 3.56q_0^2.
\end{equation*}
This means that in the integer division of $c_0$ by $q_0$, we have
\begin{equation*}
c_0=3q_0+r_1,\text{ where }0.49q_0\leq r_1\leq 0.56q_0.
\end{equation*}
This allows us to calculate the integer division of $b+1$ by $q=q_0-1$ as
\begin{equation*}
b+1=c_0q_0+r_0=c_0(q_0-1)+3q_0+r_0+r_1=(c_0+3)(q_0-1)+(r_0+r_1+3),
\end{equation*}
where the remainder $r=r_0+r_1+3$ satisfies
\begin{equation*}
0.49(q_0-1)\leq r_1\leq r\leq 0.38q_0+0.56q_0+3=0.94q_0+3\leq 0.97(q_0-1).
\end{equation*}
This means that $c+r-1\geq 2.64q_0+0.49(q_0-1)-1>3(q_0-1)$. Therefore,
\begin{equation*}
b+c=b+1+c-1=(c_0+3)(q_0-1)+c+r-1>(c_0+6)(q_0-1),
\end{equation*}
so at least three multiples of $(q_0-1)$, namely $(c_0+4)(q_0-1)$, $(c_0+5)(q_0-1)$, and $(c_0+6)(q_0-1)$, lie strictly between $b+1$ and $b+c$. This means that $x\leq c_0+4$ and $y\geq c_0+6$, so $y-x\geq 2$. This proves our claim that one of $q=q_0$ or $q=q_0-1$ will have $y-x\geq 2$.\\

Now taking either $q=q_0$ or $q=q_0-1$ so that $y-x\geq 2$, we have
\begin{equation*}
x=\left\lceil\frac{b+1}q\right\rceil\geq\frac b{q_0}\geq\sqrt{3.5}\sqrt b>\sqrt{3.4}\sqrt b\geq c,
\end{equation*}
so $x\geq c+1$ because $x$ is an integer. Because $\sqrt b\geq 270$, we have 
\begin{equation*}q\geq q_0-1\geq \frac{\sqrt b}{\sqrt{3.5}}-2\geq\frac{\sqrt b}{\sqrt{3.8}}.\end{equation*} Finally, we have
\begin{align*}
\left\lceil\frac{x-1}{y-x}\right\rceil x&\leq\left(\frac b{2q}+1\right)\left(\frac bq+1\right)\leq \left(\frac {\sqrt{3.8}b}{2\sqrt b}+1\right)\left(\frac{\sqrt{3.8}b}{\sqrt b}+1\right)\\&\leq(0.975\sqrt b+1)(1.95\sqrt b+1)\leq (0.979\sqrt b)(1.954\sqrt b)\leq 1.92b\leq 2b+c+1.
\end{align*}
\end{case}
This completes the proof.
\end{proof}

Now we can prove the remaining Part (3) of Theorem \ref{thm:main2}, the case where $2c\leq b\leq\frac{c^2}2$.

\begin{proof}[Proof of Part (3) of Theorem \ref{thm:main2}. ]
If $2\leq c\leq 499$, then Lemma \ref{lem:smallc} tells us that there is an integer $q$ and a partition $\lambda$ of $n$ with all parts in the corresponding interval $J=\{x,\ldots,y\}$, where $x$ and $y$ are as in \eqref{eq:xy}. If $c\geq 500$, then Lemma \ref{lem:analysis} and Lemma \ref{lem:frobint} again give us these $q$ and $\lambda\vdash n$. Either way, now Lemma \ref{lem:q} tells us that $G$ is missing a connected partition of this type $\lambda\vdash n$. 
\end{proof}

We also prove that spiders with four legs are not $e$-positive.

\begin{proof}[Proof of Theorem \ref{thm:spider4}. ]
Let $S$ be a spider with four legs, so that $c\geq c_1+1\geq 2$. If $b\leq \frac{c^2}2$, then the result follows from Parts (1), (2), and (3) of Theorem \ref{thm:main2}. If $b\geq\frac{c^2}2$, then 
\begin{equation*}n\geq 2b+c+1\geq c^2+c+1,\end{equation*}
and Zheng showed \cite[Corollary 4.6]{spiderepos} that such a spider $S$ is not $e$-positive.
\end{proof}

In order to prove Conjecture \ref{conj:four}, it now remains to consider trees with a vertex of degree $4$ that is adjacent to a leaf, such as the spider $S(6,4,1,1)$ from Figure \ref{fig:conpar}. It turns out that $S(6,4,1,1)$ has a connected partition of every type $\lambda\vdash 13$, which shows that Wolfgang's necessary condition will not be enough. We conclude by showing that there are infinitely many such examples.

\begin{proposition}\label{prop:6mhas}
Let $m\geq 1$ be an integer. The spider $S(6m,6m-2,1,1)$ has a connected partition of type $\lambda$ for every $\lambda\vdash 12m+1$.
\end{proposition}

Note that by Theorem \ref{thm:spider4}, the spider $S(6m,6m-2,1,1)$ is not $e$-positive.

\begin{proof}
We claim that if $\lambda=(\lambda_1,\ldots,\lambda_\ell)$ has a rearrangement $\alpha$ with $6m-1,6m\notin\text{set}(\alpha)$, then $S(6m,6m-2,1,1)$ has a connected partition of type $\lambda$. Label the vertices on the leg of length $6m-2$ as $u_1,\ldots,u_{6m-2}$, moving in from the leaf, and label the vertices on the leg of length $6m$ as $v_1,\ldots,v_{6m}$ similarly. Suppose that $\alpha$ is a rearrangement of $\lambda$ that satisfies
\begin{equation*}\alpha_1+\cdots+\alpha_i\leq 6m-2\text{ and }\alpha_1+\cdots+\alpha_{i+1}\geq 6m+1.\end{equation*}
Then we can partition the vertices $\{u_1,u_2,\ldots,u_{\alpha_1+\cdots+\alpha_i}\}$ into paths of orders $\alpha_1,\ldots,\alpha_i$, we can partition the vertices $\{v_1,v_2,\ldots,v_{\alpha_{i+2}+\cdots+\alpha_\ell}\}$ into paths of orders $\alpha_{i+2},\ldots,\alpha_\ell$, and the remaining vertices form a connected set of order $\alpha_{i+1}$. This establishes our claim.\\

Note that if $\lambda$ has at least two $1$'s or if $\lambda=(2,2,\ldots,2,1)$, then $S(6m,6m-2,1,1)$ has a connected partition of type $\lambda$, so we may assume that $\lambda$ has at most one $1$ and has a part $x\geq 3$. By our claim, it suffices to find a rearrangement $\alpha$ of $\lambda$ with $6m-1,6m\notin\text{set}(\alpha)$. Fix a rearrangement $\beta$ of $\lambda$. We may assume that $6m-1$ or $6m\in\text{set}(\beta)$, otherwise we could take $\alpha=\beta$. Letting $\beta^*=(\beta_\ell,\ldots,\beta_1)$ denote the reverse composition, we may assume that $6m-1$ or $6m\in\text{set}(\beta^*)$, otherwise we could take $\alpha=\beta^*$. Now there are only a few cases to consider.\\

If $6m\in\text{set}(\beta)$ and $6m\in\text{set}(\beta^*)$, we must have $\beta_1+\cdots+\beta_i=6m$, $\beta_{i+1}=1$,  and $\beta_{i+2}+\cdots+\beta_\ell=6m$ for some $i$. By reversing, we may assume without loss of generality that some part $\beta_k=x\geq 3$ for $1\leq k\leq i$. Then the composition
\begin{equation*}
\alpha=(\beta_1,\ldots,\beta_{k-1},\beta_{k+1},\ldots,\beta_i,1,x,\beta_{i+2},\ldots,\beta_\ell)
\end{equation*}
has $\text{set}(\alpha)=\{\ldots,6m-x,6m-x+1,6m+1,\ldots\}$, so $6m-1,6m\notin\text{set}(\alpha)$.\\

If $6m-1\in\text{set}(\beta)$, $6m\notin\text{set}(\beta)$, and $6m\in\text{set}(\beta^*)$, we must have $\beta_1+\cdots+\beta_i=6m-1$, $\beta_{i+1}=2$, and $\beta_{i+2}+\cdots+\beta_\ell=6m$ for some $i$. If there is some $\beta_k=x\geq 3$ for $1\leq k\leq i$, then the composition
\begin{equation*}
\alpha=(\beta_1,\ldots,\beta_{k-1},\beta_{k+1},\ldots,\beta_i,2,x,\beta_{i+2},\ldots,\beta_\ell)
\end{equation*}
has $\text{set}(\alpha)=\{\ldots,6m-x-1,6m-x+1,6m+1,\ldots\}$, so $6m-1,6m\notin\text{set}(\alpha)$. Otherwise, the parts $(\beta_1,\ldots,\beta_i)$ consist of a single $1$ and $(3m-1)$ $2$'s, some $\beta_k=x\geq 3$ for $i+2\leq k\leq \ell$, and the composition
\begin{equation*}
\alpha=(2,2,\ldots,2,2,x,1,2,\beta_{i+2},\ldots,\beta_{k-1},\beta_{k+1},\ldots,\beta_\ell)
\end{equation*}
has $\text{set}(\alpha)=\{\ldots,6m-2,6m-2+x,6m-1+x,\ldots\}$, so $6m-1,6m\notin\text{set}(\alpha)$. If $6m\in\text{set}(\beta)$, $6m-1\in\text{set}(\beta^*)$, and $6m\notin\text{set}(\beta^*)$, we can reverse $\beta$ to reduce to the previous case.\\

Finally, if $6m-1\in\text{set}(\beta)$, $6m\notin\text{set}(\beta)$, $6m-1\in\text{set}(\beta^*)$, and $6m\notin\text{set}(\beta^*)$, then we must have $\beta_1+\cdots+\beta_i=6m-1$, $\beta_{i+1}=3$, and $\beta_{i+2}+\cdots+\beta_\ell=6m-1$ for some $i$. If $\beta_k=1$ for some $1\leq k\leq i$, then the composition
\begin{equation*}
\alpha=(\beta_1,\ldots,\beta_{k-1},\beta_{k+1},\ldots,\beta_i,3,1,\beta_{i+2},\ldots,\beta_\ell)
\end{equation*}
has $\text{set}(\alpha)=\{\ldots,6m-2,6m+1,6m+2,\ldots\}$, so $6m-1,6m\notin\text{set}(\alpha)$. If $\beta_k=y\geq 4$ for some $1\leq k\leq i$, then the composition
\begin{equation*}
\alpha=(\beta_1,\ldots,\beta_{k-1},\beta_{k+1},\ldots,\beta_i,3,y,\beta_{i+2},\ldots,\beta_\ell)
\end{equation*}
has $\text{set}(\alpha)=\{\ldots,6m-1-y,6m+2-y,6m+2,\ldots\}$, so $6m-1,6m\notin\text{set}(\alpha)$. Now we may assume that every $1\leq k\leq i$ has $\beta_k=2$ or $3$, and by reversing, that every $i+2\leq k\leq\ell$ also has $\beta_k=2$ or $3$. Because $6m-1$ is not divisible by either $2$ or $3$, there must be some $1\leq k\leq i$ with $\beta_k=3$ and some $i+2\leq k'\leq\ell$ with $\beta_{k'}=2$. Now the composition
\begin{align*}
\alpha=(\beta_1,\ldots,\beta_{k-1},\beta_{k+1},\ldots,\beta_i,2,3,3,\beta_{i+2},\ldots,\beta_{k'-1},\beta_{k'+1},\ldots,\beta_\ell)
\end{align*}
has $\text{set}(\alpha)=\{\ldots,6m-4,6m-2,6m+1,6m+4,\ldots\}$, so $6m-1,6m\notin\text{set}(\alpha)$. This completes the proof.
\end{proof}

\section{Appendix}\label{sec:appendix}

\begin{method}\label{method:40} This is used to check that for integers $b$, $c$, and $n$ with $2\leq c\leq 40$, $2c\leq b\leq\frac{c^2}2$, and $2b+c+1\leq n\leq\left\lceil\frac b{c-1}\right\rceil(b+1)$, we can find a positive integer $q$ such that the integers $x$ and $y$ as in \eqref{eq:xy} satisfy $x\geq c+1$, $x<y$, and $tx\leq n\leq ty$ for some $t$; and therefore that there exists a partition $\lambda$ of $n$ with all parts in the interval $J=\{x,\ldots,y\}$.
\begin{lstlisting}
def checkc40():
	for c in range(2,41):
		for b in range(2*c,floor(c^2/2)+1):
			for n in range(2*b+c+1,ceil(b/(c-1))*(b+1)+1):
				print('b:',b,'c:',c,'n:',n)
				foundq=false
				for q in range(1,ceil((b+1)/c)):
					x=ceil((b+1)/q)
					y=floor((b+c)/q)
					t=floor(n/x)
					if x>=c+1 and x<y and t*x<=n<=t*y:
						print('q:',q,'x:',x,'y:',y,'t:',t)
						foundq=true
			if not foundq:
				return false
	return true
\end{lstlisting}
\end{method}

\begin{method}\label{method:500} This is used to check that for integers $b$ and $c$ with $2\leq c\leq 40$ and $2c\leq b\leq\frac{c^2}2$, we can find a positive integer $q$ such that the integers $x$ and $y$ as in \eqref{eq:xy} satisfy $x\geq c+1$, $x<y$, and the inequality \eqref{eq:analysisbound}; and therefore that there exists a partition $\lambda$ of $n$ with all parts in the interval $J=\{x,\ldots,y\}$.
\begin{lstlisting} 
def checkc500():
	for c in range(41,500):
		for b in range(2*c,floor(c^2/2)+1):
			print('b:',b,'c:',c)
			foundq=false
			for q in range(1,ceil((b+1)/c)):
				x=ceil((b+1)/q)
				y=floor((b+c)/q)
				M=ceil((x-1)/(y-x))*x
				if x>=c+1 and x<y and M<=2*b+c+1:
					print('q:',q,'x:',x,'y:',y,'M:',M)
					foundq=true
			if not foundq:
				return false
	return true

\end{lstlisting}
\end{method}

\section{Acknowledgments}
The author would like to thank Jos\'e Aliste-Prieto, Richard Stanley, Stephanie van Willigenburg, and Kai Zheng for helpful discussions.

%

\printbibliography
\end{document}